\documentclass[12pt]{amsart}

\usepackage{amsmath, amssymb, amsthm}
\usepackage{hyperref}
\usepackage{microtype}

\newtheorem{theorem}{Theorem}[section]
\newtheorem{proposition}[theorem]{Proposition}
\newtheorem{lemma}[theorem]{Lemma}
\newtheorem{corollary}[theorem]{Corollary}
\theoremstyle{remark}
\newtheorem{remark}[theorem]{Remark}

\newcommand{\Z}{\mathbb{Z}}

\newcommand{\Rr}[1]{\mathcal{R}(#1)}

\begin{document}

\title[Non-isomorphism of Integer-Coefficient Power Series Rings]{Non-isomorphism of rings of integer-coefficient holomorphic
functions on disks of varying radius}

\author{Jon Bannon \and David Feldman}

\date{}

\begin{abstract}
For $\rho \in (0,1]$, let $\Rr{\rho} = \Z[[z]] \cap \mathcal{O}(B(0,\rho))$
denote the ring of power series with integer Taylor coefficients converging
on the open disk $B(0,\rho)$.  We prove that these rings are pairwise
non-isomorphic as abstract rings.  Three ingredients drive the proof:
the ideal $(z)$ is the unique principal ideal with quotient $\Z$, so
any isomorphism sends $z$ to a generator $g$ of $(z)$; every isomorphism
is substitution by $g$, because it respects the $(z)$-adic filtration;
and a Hadamard gap series with a natural boundary at $|w| = \rho_1$
forces the image $g(B(0,\rho_2))$ into $B(0,\rho_1)$, after which the
Schwarz lemma and integrality of coefficients force $g = \pm z$ and
$\rho_1 = \rho_2$.
\end{abstract}

\maketitle

\section{Introduction}

For $\rho \in (0,\infty)$, define
\[
  \Rr{\rho} = \Z[[z]] \cap \mathcal{O}(B(0,\rho))
\]
to be the ring of formal power series $\sum_{n\geq 0} a_n z^n$ with
$a_n \in \Z$ for all $n$, converging on $B(0,\rho)$.  For $\rho > 1$
the Hadamard formula forces $a_n = 0$ eventually, so $\Rr{\rho} = \Z[z]$
for $\rho > 1$.  We restrict to $\rho \in (0,1]$ for the remainder of
the paper.

\subsection*{The non-isomorphism question}

At first sight, non-isomorphism for different $\rho$ seems clear: the
biholomorphism $B(0,\rho_1) \cong B(0,\rho_2)$ does not preserve integer
coefficients.  But these are not topological rings, and ring homomorphisms
need not be continuous.  A ring isomorphism $\Rr{\rho_1} \to \Rr{\rho_2}$
could in principle behave wildly on infinite series.  Moreover, the
zero-set geometry of $\Rr{\rho}$ is, up to scaling, independent of
$\rho$~\cite{BannonFeldman}.  Non-isomorphism requires proof.

\subsection*{Strategy}

Any ring isomorphism $\phi \colon \Rr{\rho_1} \to \Rr{\rho_2}$ sends
the ideal $(z)$ to itself (it is the unique principal ideal with
quotient $\Z$), so $g := \phi(z) = zu$ for a unit $u$.  The first key
step is that $\phi$ is then \emph{substitution by $g$}: $\phi(f) =
f\circ g$ for every $f$.  This is not a continuity hypothesis but a
consequence of $\phi$ respecting the $(z)$-adic filtration, since
$\phi((z)^N) = (g)^N = (z)^N$ and $\bigcap_N (z)^N = 0$.  In particular
$g$ and $h := \phi^{-1}(z)$ are compositional inverses.

The second step controls the image of $g$.  Using a Hadamard gap
series $f \in \Rr{\rho_1}$ whose circle $|w| = \rho_1$ is a natural
boundary, the requirement that $f \circ g = \phi(f)$ converge on
$B(0,\rho_2)$ forces $g(B(0,\rho_2)) \subseteq B(0,\rho_1)$: a value
$g(z_1)$ of modulus $> \rho_1$ would continue $f$ across its natural
boundary.  With this containment (and its mirror for $h$), $g$ is a
disk biholomorphism fixing $0$, so the Schwarz lemma makes it linear;
integrality of the coefficients of $g$ and $g^{-1}$ then forces $g =
\pm z$ and $\rho_1 = \rho_2$.

\section{The ideal $(z)$ and substitution}
\label{sec:z}

\begin{lemma}[Characterization of $(z)$]
\label{lem:z_ideal}
The ideal $(z)$ is the unique principal ideal $I \subset \Rr{\rho}$
with $\Rr{\rho}/I \cong \Z$.
\end{lemma}

\begin{proof}
Evaluation at $0$ gives $\Rr{\rho}/(z) \cong \Z$.

Conversely, let $\pi \colon \Rr{\rho} \twoheadrightarrow \Z$ be a
surjective ring homomorphism with principal kernel $(f)$.  Since $\Z$
is the prime subring of both rings, $\pi$ fixes each integer.  Let
$m = \pi(z) \in \Z$.

\emph{Step 1.}  The element $h = \sum_{n\geq 0} z^n$ belongs to
$\Rr{\rho}$ (since the radius of convergence is $1$ and $\rho \leq
1$) and satisfies $(1-z)h = 1$.  Applying $\pi$: $(1-m)\pi(h) = 1$
in $\Z$, so $m \in \{0, 2\}$.

\emph{Step 2.}  Similarly $k = \sum_{n\geq 0} z^{2n} \in \Rr{\rho}$
satisfies $(1-z^2)k = 1$, giving $(1-m^2)\pi(k) = 1$.  Since $1-4
= -3 \notin \Z^\times$, $m \neq 2$.

Hence $m = 0$ and $\pi = \mathrm{ev}_0$, so $(f) = \ker\mathrm{ev}_0
= (z)$.
\end{proof}

\begin{lemma}[Isomorphisms are substitutions]
\label{lem:substitution}
Let $\phi \colon \Rr{\rho_1} \to \Rr{\rho_2}$ be a ring isomorphism,
and set $g = \phi(z)$.  Then $(g) = (z)$ in $\Rr{\rho_2}$, so $g = zu$
for a unit $u$ with $u(0) = \pm 1$; and for every $f = \sum_n a_n z^n
\in \Rr{\rho_1}$,
\[
  \phi(f) = \sum_{n\geq 0} a_n g^n
\]
as elements of $\Rr{\rho_2}$ (equivalently, as formal power series in
$\Z[[z]]$).  In particular, writing $h = \phi^{-1}(z)$, the holomorphic
functions $g, h$ satisfy $g(h(w)) = w$ and $h(g(z)) = z$ wherever both
sides are defined.
\end{lemma}

\begin{proof}
By Lemma~\ref{lem:z_ideal}, $\phi((z)) = (z)$, so $g = zu$ with $u
\in \Rr{\rho_2}^\times$, $u(0) = \pm 1$; in particular $(g) = (z)$ and
$g^n \in (z)^n$.

Fix $f = \sum_n a_n z^n \in \Rr{\rho_1}$ and $N \geq 1$.  Since $f -
\sum_{n<N} a_n z^n \in (z)^N$, write it as $z^N \tilde f$ with $\tilde
f \in \Rr{\rho_1}$.  Applying the ring homomorphism $\phi$ (which fixes
$\Z$, so $\phi(\sum_{n<N} a_n z^n) = \sum_{n<N} a_n g^n$):
\[
  \phi(f) - \sum_{n<N} a_n g^n = \phi(z^N \tilde f) = g^N \phi(\tilde f)
  \in (g)^N = (z)^N.
\]
Thus $\phi(f) \equiv \sum_{n<N} a_n g^n \pmod{(z)^N}$ for every $N$.
The formal power series $\sum_n a_n g^n$ (well-defined in $\Z[[z]]$
since $g^n \in (z)^n$) therefore agrees with the power series of
$\phi(f)$ modulo $(z)^N$ for all $N$.  As $\bigcap_N (z)^N = \{0\}$
in $\Z[[z]]$ (a power series divisible by $z^N$ for all $N$ is zero),
we conclude $\phi(f) = \sum_n a_n g^n$ in $\Z[[z]]$, hence in
$\Rr{\rho_2}$.

Applying this to $\phi^{-1}$ (with $\phi^{-1}(z) = h$): $\phi^{-1}(F)
= \sum_n c_n h^n = F(h)$ for $F = \sum_n c_n z^n \in \Rr{\rho_2}$.
Taking $F = g$: $z = \phi^{-1}(\phi(z)) = \phi^{-1}(g) = g(h)$ as
formal power series, so $g(h(w)) = w$; symmetrically $h(g(z)) = z$.
Both identities hold as formal power series, hence as holomorphic
functions on the disks where the compositions are defined.
\end{proof}

\begin{lemma}[A natural boundary at radius $\rho$]
\label{lem:gap}
For every $\rho \in (0,1]$ there exists $f \in \Rr{\rho}$ whose radius
of convergence is exactly $\rho$ and for which the circle $|w| = \rho$
is a natural boundary: $f$ admits no analytic continuation to any
neighborhood of any point of $\{|w| = \rho\}$.
\end{lemma}

\begin{proof}
Put $f(w) = \sum_{n\geq 0} c_n w^{2^n}$ with $c_n = \lfloor \rho^{-2^n}
\rfloor \in \Z_{\geq 1}$.  Then $|c_n|^{1/2^n} \to \rho^{-1}$, so the
radius of convergence is $\rho$, whence $f \in \Rr{\rho}$.  The nonzero
terms occur at exponents $2^n$, which satisfy the Hadamard gap
condition $2^{n+1}/2^n = 2 > 1$; by the Hadamard gap
theorem~\cite{Remmert}, $|w| = \rho$ is a natural boundary.
\end{proof}

\section{Main result}
\label{sec:main}

\begin{proposition}
\label{prop:z_fixed}
Let $\phi \colon \Rr{\rho_1} \to \Rr{\rho_2}$ be a ring isomorphism.
Then $\rho_1 = \rho_2$ and $\phi(z) = \pm z$.
\end{proposition}

\begin{proof}
By Lemma~\ref{lem:substitution}, $g = \phi(z) = zu$ and $h = \phi^{-1}(z)
= zv$ with $u,v$ units, $u(0)=v(0)=\pm1$, and $g,h$ are formal
compositional inverses in $\Z[[z]]$:
\[
  h(g(z)) = z, \qquad g(h(w)) = w.
\]
Moreover, by Lemma~\ref{lem:substitution}, $\phi(f) = f\circ g$ for
every $f \in \Rr{\rho_1}$, and $\phi^{-1}(F) = F\circ h$ for every $F
\in \Rr{\rho_2}$.

\textbf{Step 1: $g(B(0,\rho_2)) \subseteq B(0,\rho_1)$.}

Since $g$ is nonconstant, $g(B(0,\rho_2))$ is open (open mapping
theorem).  If $g(B(0,\rho_2)) \not\subseteq B(0,\rho_1)$, then, being
open, $g(B(0,\rho_2))$ cannot be contained in the closed disk
$\overline{B(0,\rho_1)}$ either (an open set meeting $\{|w| = \rho_1\}$
contains points of modulus $> \rho_1$); so there is $z_1 \in
B(0,\rho_2)$ with $|g(z_1)| > \rho_1$.  Since critical points of $g$
are isolated, we may take $g'(z_1) \neq 0$.  Set $w_1 = g(z_1)$, so
$|w_1| > \rho_1$.

Let $f \in \Rr{\rho_1}$ be the gap series of Lemma~\ref{lem:gap}, whose
circle of convergence $|w| = \rho_1$ is a natural boundary.  By
Lemma~\ref{lem:substitution}, $\phi(f) = f \circ g \in \Rr{\rho_2}$,
so $f\circ g$ is holomorphic on $B(0,\rho_2)$, in particular near
$z_1$.  As $g'(z_1) \neq 0$, $g$ has a holomorphic local inverse
$g^{-1}$ near $w_1$, and near $w_1$
\[
  f(w) = (f\circ g)\big(g^{-1}(w)\big)
\]
is holomorphic.  Thus $f$ extends holomorphically to a neighborhood
of $w_1$ with $|w_1| > \rho_1$, contradicting that $f$ has radius of
convergence exactly $\rho_1$.  Hence $g(B(0,\rho_2)) \subseteq
B(0,\rho_1)$.

The identical argument applied to $\phi^{-1}$ gives $h(B(0,\rho_1))
\subseteq B(0,\rho_2)$.

\textbf{Step 2: $g$ is a biholomorphism $B(0,\rho_2) \to B(0,\rho_1)$.}

By Step 1 both composites $h \circ g$ and $g \circ h$ are defined and
holomorphic, mapping $B(0,\rho_2) \to B(0,\rho_2)$ and $B(0,\rho_1) \to
B(0,\rho_1)$ respectively, and by the compositional-inverse identities
they equal the respective identity maps.  Hence $g$ is a
biholomorphism of $B(0,\rho_2)$ onto $B(0,\rho_1)$ with inverse $h$.

\textbf{Step 3: $g(z) = \pm z$ and $\rho_1 = \rho_2$.}

The scaled map $G(w) = g(\rho_2 w)/\rho_1 \colon \mathbb{D} \to
\mathbb{D}$ is a biholomorphic automorphism fixing $0$, hence a
rotation by the Schwarz lemma: $G(w) = e^{i\theta} w$.  Therefore
$g(z) = e^{i\theta}(\rho_1/\rho_2)\, z = \lambda z$ with $\lambda =
e^{i\theta}\rho_1/\rho_2$.

Since $g \in \Rr{\rho_2}$ has integer Taylor coefficients, $\lambda
\in \Z$; its inverse $h(w) = w/\lambda \in \Rr{\rho_1}$ has integer
coefficients, so $1/\lambda \in \Z$.  Hence $\lambda = \pm 1$, and
$g(z) = \pm z$.  Then $g$ maps $B(0,\rho_2)$ onto $B(0,\rho_2)$; as
this image equals $B(0,\rho_1)$ by Step 2, $\rho_1 = \rho_2$.  Thus
$\phi(z) = \pm z$ and $\rho_1 = \rho_2$.
\end{proof}

\begin{theorem}
\label{thm:main}
For $\rho_1, \rho_2 \in (0,1]$, $\Rr{\rho_1} \cong \Rr{\rho_2}$ if and
only if $\rho_1 = \rho_2$.
\end{theorem}

\begin{proof}
If $\rho_1 = \rho_2$ the rings are identical.  Conversely, any isomorphism
satisfies $\rho_1 = \rho_2$ by Proposition~\ref{prop:z_fixed}.
\end{proof}

\section{The $p$-th root reconstruction}
\label{sec:pth}

The main theorem is proved above.  We give an independent reconstruction
of $\rho$ from the ring structure, showing $\rho$ is explicitly recoverable.

\begin{proposition}
\label{prop:pth_root}
Let $p$ be prime and $n, k \geq 1$.  There exists $h \in \Rr{\rho}$
with $h(0) = 1$ and $h^p = 1 - p^n z^k$ if and only if $\rho \leq
p^{-n/k}$.
\end{proposition}

\begin{proof}
The unique formal power series solution is
\begin{align*}
  h(z) &= (1-p^n z^k)^{1/p} = \sum_{m=0}^\infty \binom{1/p}{m}(-1)^m
  p^{nm} z^{km}, \\
  c_m &= (-1)^m \cdot
  \frac{p^{(n-1)m}\prod_{j=0}^{m-1}(1-pj)}{m!}.
\end{align*}

\emph{Integrality.}  We show $c_m \in \Z_\ell$ for every prime $\ell$,
whence $c_m \in \bigcap_\ell \Z_\ell = \Z$.

For $\ell = p$: by the $p$-adic binomial theorem~\cite[Ch.~IV]{Koblitz},
$(1-x)^{1/p} \in \Z_p[[x]]$ for $x \in p\Z_p$; applying with $x =
p^n z^k$ (valid since $n \geq 1$) gives $c_m \in \Z_p$.

For $\ell \neq p$: here $p$ is a unit in $\Z_\ell$, so $1/p \in \Z_\ell$.
The generalized binomial coefficient $\binom{t}{m} = t(t-1)\cdots(t-m+1)
/m!$ is a $\Z_\ell$-valued function of $t \in \Z_\ell$ (it is the
uniform limit of $\binom{a}{m} \in \Z$ over integers $a \to t$ in
$\Z_\ell$, and $\Z_\ell$ is complete).  Hence $\binom{1/p}{m} \in
\Z_\ell$, and $c_m = (-1)^m p^{nm}\binom{1/p}{m} \in \Z_\ell$ since
$p^{nm} \in \Z_\ell$.

\emph{Radius.}  The binomial series $(1-w)^{1/p}$ has radius of
convergence $1$ in $w$, so $(1-p^n z^k)^{1/p}$ has radius $p^{-n/k}$.
Hence $h \in \Rr{\rho}$ iff $\rho \leq p^{-n/k}$.
\end{proof}

\begin{corollary}
Setting $I(\Rr{\rho}) = \{(p,n,k) : (1-p^n z^k)^{1/p} \in \Rr{\rho}\}$,
we have $\rho = \inf\{p^{-n/k} : (p,n,k) \in I(\Rr{\rho})\}$, so
$I(\Rr{\rho})$ determines $\rho$ uniquely.  Since any isomorphism satisfies
$\phi(z) = \pm z$ and hence $\phi(1-p^n z^k) = 1 \mp p^n z^k$ (with the
same radius threshold $p^{-n/k}$), $I$ is an isomorphism invariant.
\end{corollary}

\begin{remark}[Automorphism group]
Every ring automorphism of $\Rr{\rho}$ is $z \mapsto \pm z$, giving
$\mathrm{Aut}(\Rr{\rho}) \cong \Z/2\Z$.
\end{remark}

\begin{remark}[Open questions]
(i)~Are the rings $\Rr{\rho}$ for distinct $\rho$ elementarily equivalent?
(ii)~Do there exist ring homomorphisms $\Rr{\rho_1} \to \Rr{\rho_2}$ for
$\rho_1 \neq \rho_2$?  (iii)~Is $\Rr{\rho}$ a B\'ezout domain?
\end{remark}

\section*{Declaration on the Use of Artificial Intelligence}
In preparing this paper, the authors made use of AI language model
assistance (Claude) as a tool in the writing process---including
\LaTeX{} preparation and editing---and in exploratory work, such as
proof-checking and verifying arguments. All mathematical content and
results have been independently reviewed and verified by the authors,
who take full responsibility for the correctness, originality, and
presentation of this work. No AI system is an author of this paper,
nor was any AI system used to generate mathematical claims that were
not subsequently checked by the authors.

\end{document}